\documentclass{amsart}

\usepackage{ amsthm } 
\usepackage{ array }
\usepackage{ booktabs }
\usepackage{ url }

\theoremstyle{definition}

\newtheorem{definition}{Definition}[section]
\newtheorem{notation}[definition]{Notation}

\newtheorem{remark}[definition]{Remark}
\newtheorem{algorithm}[definition]{Algorithm}

\theoremstyle{plain}

\newtheorem{lemma}[definition]{Lemma}
\newtheorem{proposition}[definition]{Proposition}
\newtheorem{theorem}[definition]{Theorem}

\allowdisplaybreaks

\begin{document}

\title[Deformations of the diassociative and dendriform operads]
{One-parameter deformations of the diassociative and dendriform operads}

\author{Murray R. Bremner}

\address{Department of Mathematics and Statistics, University of Saskatchewan, Canada}

\email{bremner@math.usask.ca}

\subjclass[2010]{Primary 17A30. Secondary 15A21, 15A54, 16S80, 17A50, 17B63, 18D50.}

\keywords{Diassociative algebras, dendriform algebras, Hermite normal form, LLL algorithm, polarization,
deformation, Poisson algebras, algebraic operads, Koszul duality.}

\thanks{The author was supported by a Discovery Grant from NSERC, the Natural Sciences 
and Engineering Research Council of Canada.}

\begin{abstract}
Livernet and Loday constructed a polarization of the nonsymmetric associative operad $\mathcal{A}$ 
with one operation into a symmetric operad $\mathcal{SA}$ with two operations (the Lie bracket and 
Jordan product), and defined a one-parameter deformation of $\mathcal{SA}$ which includes Poisson 
algebras.
We combine this with the dendriform splitting of an associative operation into the sum of two nonassociative 
operations, and use Koszul duality for quadratic operads, to construct one-parameter deformations 
of the nonsymmetric dendriform and diassociative operads into the category of symmetric operads.
\end{abstract}

\maketitle


\section{The dendriform and diassociative operads}

We write $\mathcal{O}_1$ for the free nonsymmetric operad generated by one binary operation $(a,b) \mapsto ab$, and
$\mathcal{O}_2$ for the free nonsymmetric operad generated by left and right binary operations $(a,b) \mapsto a \prec b$ 
and $(a,b) \mapsto a \succ b$.

\begin{definition} \label{defsplit}
The operad morphism $\mathsf{split}\colon \mathcal{O}_1 \to \mathcal{O}_2$,
which is injective but not surjective, is determined by its value on the generator:
  \[
  \mathsf{split}( ab ) = a \prec b + a \succ b.
  \]
\end{definition}

\begin{definition}
The (nonsymmetric) \textbf{associative operad} $\mathcal{A}$ is the quotient of 
$\mathcal{O}_1$ by the ideal generated by the associator $\alpha = (ab)c - a(bc)$.
Since $\alpha = 0$ is a quadratic relation (every term contains two operations), $\mathcal{A}$ is a
quadratic operad.
\end{definition}

\begin{lemma} \label{splitassociator}
The image of the associator under the splitting morphism is:
  \[
  \begin{array}{r}
  \mathsf{split}( \alpha )
  =
    ( a \prec b ) \prec c 
  \;+\; ( a \prec b ) \succ c
  \;+\; ( a \succ b ) \prec c 
  \;+\; ( a \succ b ) \succ c
  \\
  \;-\; a \prec ( b \prec c )
  \;-\; a \prec ( b \succ c )
  \;-\; a \succ ( b \prec c )
  \;-\; a \succ ( b \succ c ).
  \end{array}
  \]
\end{lemma}

\begin{definition} \label{dendsplit}
We rearrange $\mathsf{split}( \alpha )$ into the \textbf{dendriform splitting} of $\alpha$:
  \[
  \begin{array}{r@{\,}l}
  &
  [ \;
  ( a \prec b ) \prec c \;-\; a \prec ( b \prec c ) \;-\; a \prec ( b \succ c ) 
  \; ]
  \\
  +
  &
  [ \;
  ( a \prec b ) \succ c \;+\; ( a \succ b ) \succ c \;-\; a \succ ( b \succ c )
  \; ]
  \\
  +
  &
  [ \;
  ( a \succ b ) \prec c \;-\; a \succ ( b \prec c )
  \; ].
  \end{array}
  \]
\end{definition}

\begin{definition} \label{defdend}
The three relations in brackets in Definition \ref{dendsplit} define the (quadratic nonsymmetric) 
\textbf{dendriform operad} $\mathsf{Dend}$; see \cite{Aguiar2000,Loday2001,LodayRonco1998}:
  \[
  \begin{array}{l}
  ( a \prec b ) \prec c - a \prec ( b \prec c ) - a \prec ( b \succ c ) = 0,
  \\
  ( a \prec b ) \succ c + ( a \succ b ) \succ c - a \succ ( b \succ c ) = 0,
  \\
  ( a \succ b ) \prec c - a \succ ( b \prec c ) = 0.
  \end{array}
  \]
The (quadratic nonsymmetric) \textbf{diassociative operad} $\mathsf{Dias}$ is the Koszul dual of $\mathsf{Dend}$; 
see \cite{Loday1995}.
\end{definition}

\begin{lemma} \label{defdias}
\cite{Loday1995,Loday2001}
The diassociative operad $\mathsf{Dias}$ is defined by these relations:
  \[
  \begin{array}{l}
  ( a \prec b ) \prec c = a \prec ( b \prec c ) = a \prec ( b \succ c ),
  \\
  ( a \prec b ) \succ c = ( a \succ b ) \succ c = a \succ ( b \succ c ),
  \\
  ( a \succ b ) \prec c = a \succ ( b \prec c ).
  \end{array}
  \]
\end{lemma}

\begin{proof}
In the matrix of the dendriform relations in Definition \ref{defdend},
the $(i,j)$ entry is the coefficient of the $j$-th basis monomial in the $i$-th relation;
the basis monomials are ordered as in the expansion of Lemma \ref{splitassociator}:
  \[
  \left[
  \begin{array}{rrrrrrrr}
  1 &  0 &  0 &  0 & -1 & -1 &  0 &  0 \\[-1mm]
  0 &  1 &  0 &  1 &  0 &  0 &  0 & -1 \\[-1mm]
  0 &  0 &  1 &  0 &  0 &  0 & -1 &  0 
  \end{array}
  \right]
  \]
To compute the matrix of relations for the Koszul dual operad, we first
right-multiply by the block diagonal matrix $\mathrm{diag}(I_4,-I_4)$, changing
the signs of columns 5--8 corresponding to the second placement of parentheses; 
see \cite[Ch.~7]{LodayVallette2012}:  
  \[
  A =
  \left[
  \begin{array}{rrrrrrrr}
  1 &  0 &  0 &  0 &  1 &  1 &  0 &  0 \\[-1mm]
  0 &  1 &  0 &  1 &  0 &  0 &  0 &  1 \\[-1mm]
  0 &  0 &  1 &  0 &  0 &  0 &  1 &  0 
  \end{array}
  \right]
  \]
From this we compute a basis for the null space consisting of short vectors.
To find a short basis for the null space of an $m \times n$ integer matrix $A$ of rank $r$ we use 
integer Gaussian elimination to construct an invertible $n \times n$ integer matrix $U$ for which 
$UA^t$ is the HNF (Hermite normal form) of the transpose $A^t$, extract the submatrix $N$ consisting of 
the last $n{-}r$ rows of $U$, and use the LLL (Lenstra-Lenstra-Lov\'asz) algorithm for lattice basis reduction 
on the rows of $N$; see \cite{Bremner2012,BremnerPeresi2009}.
We obtain:
  \[
  N =
  \left[
  \begin{array}{rrrrrrrr}
  1 & 0 & 0 &  0 & -1 &  0 &  0 &  0 \\[-1mm]
  1 & 0 & 0 &  0 &  0 & -1 &  0 &  0 \\[-1mm]
  0 & 1 & 0 & -1 &  0 &  0 &  0 &  0 \\[-1mm]
  0 & 1 & 0 &  0 &  0 &  0 &  0 & -1 \\[-1mm]
  0 & 0 & 1 &  0 &  0 &  0 & -1 &  0
  \end{array}
  \right]
  \]
This is the coefficient matrix for the diassociative relations.
\end{proof}
  
\begin{definition}  
In Lemma \ref{defdias}, the first (second) pair of equations are \textbf{left (right) associativity} 
and the \textbf{left (right) bar identity}; the last is \textbf{inner associativity}.
\end{definition}

General references for algebraic operads are \cite{LodayVallette2012} for theory and \cite{BremnerDotsenko2016}
for algorithms.


\section{Polarization of the associative operad}

We take a different approach from \cite{LivernetLoday1998,MarklRemm2006} which emphasizes the method by which the shortest defining 
relations for the polarization of an operation may be discovered using the HNF and LLL algorithms.

\begin{definition} \label{defpolarized}
As in Definition \ref{defsplit} we write an operation $ab$ as the sum of two other operations, but now these
two operations are defined in terms of $ab$; they are the (scaled)
\textbf{Lie bracket} $[a,b]$ and the \textbf{Jordan product} $a \circ b$:
  \[
  [a,b] = \tfrac12(ab-ba), \qquad a \circ b = \tfrac12(ab+ba), \qquad ab = [a,b] + a \circ b.
  \]
The Lie bracket is anti-commutative, $[b,a] = -[a,b]$, and the Jordan product is
commutative, $b \circ a = a \circ b$; hence the use of the term \textbf{polarization}:
these two new operations are eigenvectors for the transposition of the arguments in $ab$. 
\end{definition}

\begin{notation}
We have passed from the nonsymmetric operad $\mathcal{O}_1$ to its 
symmetrization $\mathcal{SO}_1$: we have introduced permutations of the arguments.
The subspace $\mathcal{O}_1(3)$ of arity 3 in $\mathcal{O}_1$ has dimension 2 and basis $\{ (ab)c, ab(c) \}$; 
applying all 6 permutations to each monomial we see that $\mathcal{SO}_1(3)$ has dimension 12.
\end{notation}

\begin{definition}
Using the polarized operations of Definition \ref{defpolarized}, we obtain a different ordered basis of 
$\mathcal{SO}_1(3)$: (skew-)symmetry allows us to write every monomial in the form 
$( a^\sigma \circ_1 b^\sigma ) \circ_2 c^\sigma$
for $\sigma \in S_3$ where $a^\sigma$ precedes $b^\sigma$ in lex order.
The \textbf{polarized basis} for $\mathcal{SO}_1(3)$ is:
  \[
  \begin{array}{l@{\quad}l@{\quad}l@{\quad}l@{\quad}l@{\quad}l}
  [ [a,b], c ], &
  [ [a,c], b ], &
  [ [b,c], a ], &
  [ a \circ b, c ], &
  [ a \circ c, b ], &
  [ b \circ c, a ],
  \\ {}
  [a,b] \circ c, &
  [a,c] \circ b, &
  [b,c] \circ a, &
  ( a \circ b ) \circ c, &
  ( a \circ c ) \circ b, &
  ( b \circ c ) \circ a.
  \end{array}
  \]
\end{definition}

\begin{lemma} \label{lemmapolassoc}
The expansion of the associator in terms of the polarized basis is:
  \[
  \begin{array}{l}
  \alpha^\pm = 
  \big( \; [ [a,b], c ] \;+\; [ [b,c], a ] \; \big)
  \;+\; 
  \big( \; [ a \circ b, c ] \;+\; [ b \circ c, a ] \; \big)
  \\
  \;+\; 
  \big( \; [a,b] \circ c \;-\; [b,c] \circ a \; \big)
  \;+\; 
  \big( \; ( a \circ b ) \circ c \;-\; ( b \circ c ) \circ a \; \big).
  \end{array}
  \]
\end{lemma}

\begin{proof}
We expand the associator into the polarized basis using $ab = [a,b] + a \circ b$, 
and apply (anti-)\-com\-mu\-ta\-tivity to the last four terms:
  \begin{align*}
  &
  [ [a,b], c ] + [ a {\circ} b, c ] + [a,b] {\circ} c + ( a {\circ} b ) {\circ} c
  - [ a, [b,c] ] - [ a, b {\circ} c ] - a {\circ} [b,c] - a {\circ} ( b {\circ} c )
  \\
  = \; 
  &
  [ [a,b], c ] + [ a {\circ} b, c ] + [a,b] {\circ} c + ( a {\circ} b ) {\circ} c
  + [ [b,c], a ] + [ b {\circ} c, a ] - [b,c] {\circ} a - ( b {\circ} c ) {\circ} a.
  \end{align*}
We then collect terms with the same pattern of polarized operations. 
\end{proof}

\begin{definition}
We call $\alpha^\pm$ the \textbf{polarized associativity relation}.
\end{definition}

\begin{proposition} \label{propositionpolarization}
The $S_3$-submodule of $\mathcal{SO}_1(3)$ generated by the polarized associativity relation
is also generated by these two relations with only three terms each:
  \[
  [[a, c], b] + ( a \circ b ) \circ c - ( b \circ c ) \circ a,
  \qquad
  [ a \circ b, c] - [a, c] \circ b - [b, c] \circ a.
  \]
\end{proposition}

\begin{proof}
Our goal is to find the shortest integer vectors in the submodule generated by $\alpha^\pm$,
and then from these, find a minimal set of generators for the submodule.
We apply the permutations in $S_3$ to the arguments of $\alpha^\pm$ and store 
the results in the $6 \times 12$ matrix $P$ whose $(i,j)$ entry is the coefficient of the $j$-th
polarized basis monomial in the application of the $i$-th permutation (in lex order).
We obtain this matrix whose rows form a basis of the 
$S_3$-submodule of $\mathcal{SO}_1(3)$ generated by $\alpha^\pm$:
  \[
  P =
  \left[
  \begin{array}{rrrrrrrrrrrr}
   1 &  0 &  1 & 1 & 0 & 1 &  1 &  0 & -1 &  1 &  0 & -1 \\[-1mm]
   0 &  1 & -1 & 0 & 1 & 1 &  0 &  1 &  1 &  0 &  1 & -1 \\[-1mm]
  -1 &  1 &  0 & 1 & 1 & 0 & -1 & -1 &  0 &  1 & -1 &  0 \\[-1mm]
   0 & -1 &  1 & 0 & 1 & 1 &  0 &  1 &  1 &  0 & -1 &  1 \\[-1mm]
   1 & -1 &  0 & 1 & 1 & 0 & -1 & -1 &  0 & -1 &  1 &  0 \\[-1mm]
  -1 &  0 & -1 & 1 & 0 & 1 &  1 &  0 & -1 & -1 &  0 &  1
  \end{array}
  \right]
  \]
We want to find the shortest integer vectors in the row space of $P$. 
This requires two iterations of the method of the previous section to find a short
integer basis of the null space: the row space is the null space of the null space.
We first compute an invertible $12 \times 12$ integer matrix $U_1$ for which $U_1 P^t$ is the HNF
of $P^t$, and let $N_1$ be the $6 \times 12$ matrix consisting of the bottom half of $U_1$.
We then compute an invertible $12 \times 12$ integer matrix $U_2$ for which $U_2 N_1^t$ is the HNF
of $N_1^t$, and let $N_2$ be the $6 \times 12$ matrix consisting of the bottom half of $U_2$:
  \[
  N_2 =
  \left[
  \begin{array}{rrrrrrrrrrrr}
  -1 &  1 & -1 &  0 &  0 &  0 & 0 & 0 & 0 &  0 & 0 & 0 \\[-1mm]
   0 &  0 &  0 & -1 & -1 &  0 & 1 & 1 & 0 &  0 & 0 & 0 \\[-1mm]
   0 &  0 &  0 & -1 &  0 &  0 & 0 & 1 & 1 &  0 & 0 & 0 \\[-1mm]
   0 &  0 &  0 & -1 & -1 & -1 & 0 & 0 & 0 &  0 & 0 & 0 \\[-1mm]
   0 &  0 & -1 &  0 &  0 &  0 & 0 & 0 & 0 & -1 & 1 & 0 \\[-1mm]
   0 & -1 &  0 &  0 &  0 &  0 & 0 & 0 & 0 & -1 & 0 & 1
  \end{array}
  \right]
  \]
The squared Euclidean lengths of the rows of $N_2$ are 3, 4, 3, 3, 3, 3.  
If we apply the LLL algorithm for lattice basis reduction with higher reduction parameter 9/10 
(instead of the usual 3/4) then we obtain a lattice basis $N_3$ for the integer row space of $N_2$ in 
which every vector has square-length 3; we have also put these basis vectors into upper triangular form:
  \[
  N_3 =
  \left[
  \begin{array}{rrrrrrrrrrrr}
  1 & -1 & 1 & 0 & 0 & 0 & 0 & 0 & 0 & 0 & 0 & 0 \\[-1mm]
  0 & 1 & 0 & 0 & 0 & 0 & 0 & 0 & 0 & 1 & 0 & -1 \\[-1mm]
  0 & 0 & 1 & 0 & 0 & 0 & 0 & 0 & 0 & 1 & -1 & 0 \\[-1mm]
  0 & 0 & 0 & 1 & 1 & 1 & 0 & 0 & 0 & 0 & 0 & 0 \\[-1mm]
  0 & 0 & 0 & 1 & 0 & 0 & 0 & -1 & -1 & 0 & 0 & 0 \\[-1mm]
  0 & 0 & 0 & 0 & 1 & 0 & -1 & 0 & 1 & 0 & 0 & 0
  \end{array}
  \right]
  \]
We write out the relations corresponding to the rows of $N_3$, but recalling that these are basis vectors, 
where all we need is a set of module generators:
  \begin{equation}
  \label{N3basis}
  \begin{array}{l}
      [[a, b], c]  
   -  [[a, c], b]  
   +  [[b, c], a]  
  = 0,
  \\ {}
      [[a, c], b]  
   +  ( a \circ b ) \circ c  
   -  ( b \circ c ) \circ a  
  = 0,
  \\ {}
      [[b, c], a]  
   +  ( a \circ b ) \circ c  
   -  ( a \circ c ) \circ b  
  = 0,
  \\ {}
      [ a \circ b, c ]  
   +  [ a \circ c, b ]  
   +  [ b \circ c, a ]  
  = 0,
  \\ {}
      [ a \circ b, c]  
   -  [a, c] \circ b  
   -  [b, c] \circ a  
  = 0,
  \\ {}
      [ a \circ c, b ]  
   -  [a, b] \circ c  
   +  [b, c] \circ a  
  = 0.
  \end{array}
  \end{equation}
We extract a set of module generators in two stages.
First, we retain row $i$ if and only if the relation it represents does not belong to
the submodule generated by the previous rows: this leaves us with rows 1, 2, 4, 5.
Second, we remove each row from the set of four generators and compute the submodule
generated by the remaining three relations, retaining a row if and only if the three relations
generate a proper submodule; this leaves us with rows 2 and 5.
\end{proof}

\begin{definition}
The relations of Proposition \ref{propositionpolarization} are the \textbf{associator relation} and 
the \textbf{derivation relation}, since they may be written as follows:
  \[
  ( a \circ b ) \circ c  -  a \circ ( b \circ c )  
  = 
  [ [c, a], b],
  \qquad
  [ a \circ b, c]
  =  
  [a, c] \circ b + a \circ [b, c].
  \]
The first expresses the Jordan associator in terms of the Lie triple product, and the second states that the Lie bracket 
is a derivation of the Jordan product.
\end{definition}


\section{Deforming the polarization: the Poisson operad}

\begin{definition} \label{defpolsymassoperad}
The \textbf{deformation of the polarization of the symmetrized associative operad} \cite{LivernetLoday1998,MarklRemm2006} is defined by 
these three relations satisfied by an anticommutative product $[a,b]$ and a commutative product $a \circ b$:
  \[
  \begin{array}{l}
  [[a, b], c] - [[a, c], b] + [[b, c], a] = 0,
  \\ {}
  q [[a, c], b] + ( a \circ b ) \circ c - ( b \circ c ) \circ a = 0,
  \\ {}
  [ a \circ b, c] - [a, c] \circ b - [b, c] \circ a = 0.
  \end{array}
  \]
This symmetric operad $\mathcal{SA}_q$ is a module over the polynomial ring $\mathbb{F}[q]$.
\end{definition}

\begin{remark}
In Definition \ref{defpolsymassoperad}, the deformation parameter $q$ appears only in the coefficient of 
the first term of the second relation.
Otherwise, these relations are identical to 1, 2, 5 of \eqref{N3basis}.
The first relation is the Jacobi identity for Lie algebras.
If $q \ne 0$ then the first relation is the alternating sum of the second, but it must be included
for $q = 0$.
If $q = 1$ then these relations are equivalent to the associator and derivation relations of Proposition
\ref{propositionpolarization}.
\end{remark}

\begin{lemma}
The HNF of the matrix whose row module over $\mathbb{F}[q]$ is the $S_3$-module generated by the 
second and third relations in Definition \ref{defpolsymassoperad} is:
  \[
  \left[
  \begin{array}{rrrrrrrrrrrr}
  q & 0 & 0 & 0 & 0 & 0 &  0 &  0 &  0 & 0 &  1 & -1 \\
  0 & q & 0 & 0 & 0 & 0 &  0 &  0 &  0 & 1 &  0 & -1 \\
  0 & 0 & q & 0 & 0 & 0 &  0 &  0 &  0 & 1 & -1 &  0 \\
  0 & 0 & 0 & 1 & 0 & 0 &  0 & -1 & -1 & 0 &  0 &  0 \\
  0 & 0 & 0 & 0 & 1 & 0 & -1 &  0 &  1 & 0 &  0 &  0 \\
  0 & 0 & 0 & 0 & 0 & 1 &  1 &  1 &  0 & 0 &  0 &  0
  \end{array}
  \right]
  \]
If $q \ne 0$, this matrix has rank 6, but if $q = 0$, its rank is only 5.
\end{lemma}

\begin{definition}
If we set $q = 0$ in Definition \ref{defpolsymassoperad} then we obtain the relations defining 
a \textbf{Poisson algebra}, which has a commutative associative operation $a \circ b$ and 
a Lie bracket $[a,b]$ (anticommutative operation satisfying the Jacobi identity), where
the Lie bracket acts as derivations of the commutative associative product:
  \[
  \begin{array}{l}
  [[a, b], c] + [[b, c], a] + [[c, a], b] = 0,
  \\ {}
  ( a \circ b ) \circ c = ( b \circ c ) \circ a,
  \\ {}
  [ a \circ b, c] = [a, c] \circ b + a \circ [b, c].
  \end{array}
  \]
\end{definition}


\section{Splitting the polarization: deformation of dendriform}

\begin{algorithm} \label{algorithm41}
First, we expand the three relations of Definition \ref{defpolsymassoperad}, using the definitions of the Lie bracket 
and Jordan product, into the free symmetric operad $\mathcal{SO}_1$ with one binary operation (not the symmetrized 
associative operad, since we need to keep track of the placements of parentheses).

Second, we replace the binary operation in $\mathcal{SO}_1$ by the sum of left and right operations in $\mathcal{SO}_2$ 
using the splitting morphism of Definition \ref{defsplit}, and apply Definition \ref{dendsplit} to decompose each relation 
into three corresponding dendriform relations.

The results belong to $\mathcal{SO}_2(3)$, the 48-dimensional space of arity 3 relations in the 
free symmetric operad generated by two binary operations $\prec$, $\succ$.
\end{algorithm}

\begin{notation} \label{notationSO23}
An ordered basis for $\mathcal{SO}_2(3)$ consists of 8 groups of 6 elements; each group consists
of the permutations of $a,b,c$ in lex order applied to the arguments (indicated by dashes) of
the following ordered association types:
  \[
  \quad
  \begin{array}{c@{\qquad}c@{\qquad}c@{\qquad}c}
  ( - \prec - ) \prec -, &
  ( - \prec - ) \succ -, &
  ( - \succ - ) \prec -, &
  ( - \succ - ) \succ -,
  \\[-1mm]
  - \prec ( - \prec - ), &
  - \prec ( - \succ - ), &
  - \succ ( - \prec - ), &
  - \succ ( - \succ - ).  
  \end{array}
  \]
We write vectors with 48 components as $4 \times 12$ matrices in which the $(i,j)$ entry is the $k$-th entry of 
the vector for $k = 12(i{-}1) + j$.
\end{notation}

\begin{lemma} \label{lemmasplitdeform}
After applying part 1 of Algorithm \ref{algorithm41} to the relations of Definition \ref{defpolsymassoperad}
we obtain the following elements of $\mathcal{SO}_2(3)$:
  \begin{align*}
  &
  \left[
  \begin{array}{rrrrrrrrrrrr}
  1 & -1 & -1 & 1 & 1 & -1 & 1 & -1 & -1 & 1 & 1 & -1 \\
  1 & -1 & -1 & 1 & 1 & -1 & 1 & -1 & -1 & 1 & 1 & -1 \\
  -1 & 1 & 1 & -1 & -1 & 1 & -1 & 1 & 1 & -1 & -1 & 1 \\
  -1 & 1 & 1 & -1 & -1 & 1 & -1 & 1 & 1 & -1 & -1 & 1
  \end{array}
  \right]
  \\
  &
  \left[
  \begin{array}{rrrrrrrrrrrr}
  1 & q & 1 & -1 & -q & -1 & 1 & q & 1 & -1 & -q & -1 \\
  1 & q & 1 & -1 & -q & -1 & 1 & q & 1 & -1 & -q & -1 \\
  -1 & -1 & -q & q & 1 & 1 & -1 & -1 & -q & q & 1 & 1 \\
  -1 & -1 & -q & q & 1 & 1 & -1 & -1 & -q & q & 1 & 1
  \end{array}
  \right]
  \\
  &
  \left[
  \begin{array}{rrrrrrrrrrrr}
  1 & -1 & 1 & -1 & 1 & 1 & 1 & -1 & 1 & -1 & 1 & 1 \\
  1 & -1 & 1 & -1 & 1 & 1 & 1 & -1 & 1 & -1 & 1 & 1 \\
  -1 & 1 & -1 & 1 & -1 & -1 & -1 & 1 & -1 & 1 & -1 & -1 \\
  -1 & 1 & -1 & 1 & -1 & -1 & -1 & 1 & -1 & 1 & -1 & -1
  \end{array}
  \right]
  \end{align*}
\end{lemma}

\begin{definition}
We construct the $18 \times 48$ matrix $X$ in which rows 1--6, 7--12, 13--18 respectively are 
obtained by applying all permutations of the arguments $a,b,c$ to the relations of Lemma \ref{lemmasplitdeform}.
To apply part 2 of Algorithm \ref{algorithm41}, we partition $X$ into $18 \times 6$ blocks $X_1,\dots,X_8$ corresponding 
to the association types of Notation \ref{notationSO23}, and rearrange them following the dendriform splitting 
into the $54 \times 48$ matrix $Y$:
  \[
  Y =
  \left[
  \begin{array}{ccc}
  \left[ \begin{array}{ccc} X_1 & X_5 & X_6 \end{array} \right] & 0 & 0 
  \\
  0 & \left[ \begin{array}{ccc} X_2 & X_4 & X_8 \end{array} \right] & 0
  \\
  0 & 0 & \left[ \begin{array}{ccc} X_3 & X_7 \end{array} \right]
  \end{array}
  \right]
  \]
The diagonal blocks in $Y$ have sizes $18 \times 18$, $18 \times 18$, $18 \times 12$ but each has rank 6.
\end{definition}

  \begin{figure}[ht]
  \small
  \[
  \begin{array}{c}
  \left[
  \begin{array}{c@{\;}c@{\;}c@{\;}c@{\;}c@{\;}c@{\;}c@{\;}c@{\;}c@{\;}c@{\;}c@{\;}c@{\;}c@{\;}c@{\;}c@{\;}c@{\;}c@{\;}c}
  1 & . & . & . & . & 1 & {-}1 & . & . & . & . & {-}1 & {-}1 & . & . & . & . & {-}1 \\
  . & 1 & . & . & {-}1 & 1 & . & {-}1 & . & . & 1 & {-}1 & . & {-}1 & . & . & 1 & {-}1 \\
  . & . & 1 & . & 1 & . & . & . & {-}1 & . & {-}1 & . & . & . & {-}1 & . & {-}1 & . \\
  . & . & . & 1 & 1 & {-}1 & . & . & . & {-}1 & {-}1 & 1 & . & . & . & {-}1 & {-}1 & 1 \\
  . & . & . & . & q{+}3 & . & q{-}1 & {-}q{+}1 & . & {-}q{+}1 & {-}q{-}3 & q{-}1 & q{-}1 & {-}q{+}1 & . & {-}q{+}1 & {-}q{-}3 & q{-}1 \\
  . & . & . & . & . & q{+}3 & . & {-}q{+}1 & q{-}1 & {-}q{+}1 & q{-}1 & {-}q{-}3 & . & {-}q{+}1 & q{-}1 & {-}q{+}1 & q{-}1 & {-}q{-}3
  \end{array}
  \right]
  \\[11mm]
  \left[
  \begin{array}{c@{\;}c@{\;}c@{\;}c@{\;}c@{\;}c@{\;}c@{\;}c@{\;}c@{\;}c@{\;}c@{\;}c@{\;}c@{\;}c@{\;}c@{\;}c@{\;}c@{\;}c}
  1 & . & . & . & . & 1 & 1 & . & . & . & . & 1 & {-}1 & . & . & . & . & {-}1 \\
  . & 1 & . & . & {-}1 & 1 & . & 1 & . & . & {-}1 & 1 & . & {-}1 & . & . & 1 & {-}1 \\
  . & . & 1 & . & 1 & . & . & . & 1 & . & 1 & . & . & . & {-}1 & . & {-}1 & . \\
  . & . & . & 1 & 1 & {-}1 & . & . & . & 1 & 1 & {-}1 & . & . & . & {-}1 & {-}1 & 1 \\
  . & . & . & . & q{+}3 & . & . & . & . & . & q{+}3 & . & q{-}1 & {-}q{+}1 & . & {-}q{+}1 & {-}q{-}3 & q{-}1 \\
  . & . & . & . & . & q{+}3 & . & . & . & . & . & q{+}3 & . & {-}q{+}1 & q{-}1 & {-}q{+}1 & q{-}1 & {-}q{-}3
  \end{array}
  \right]
  \\[11mm]
  \left[
  \begin{array}{c@{\;}c@{\;}c@{\;}c@{\;}c@{\;}c@{\;}c@{\;}c@{\;}c@{\;}c@{\;}c@{\;}c@{\;}c@{\;}c@{\;}c@{\;}c@{\;}c@{\;}c}
  1 & . & . & . & . & 1 & {-}1 & . & . & . & . & {-}1 \\
  . & 1 & . & . & {-}1 & 1 & . & {-}1 & . & . & 1 & {-}1 \\
  . & . & 1 & . & 1 & . & . & . & {-}1 & . & {-}1 & . \\
  . & . & . & 1 & 1 & {-}1 & . & . & . & {-}1 & {-}1 & 1 \\
  . & . & . & . & q{+}3 & . & q{-}1 & {-}q{+}1 & . & {-}q{+}1 & {-}q{-}3 & q{-}1 \\
  . & . & . & . & . & q{+}3 & . & {-}q{+}1 & q{-}1 & {-}q{+}1 & q{-}1 & {-}q{-}3
  \end{array}
  \right]
  \end{array}
  \]
  \vspace{-5mm}
  \caption{Hermite normal forms of the diagonal blocks of $Y$}
  \label{HNFY}
  \end{figure}

We compute the HNFs of the diagonal blocks of $Y$; after removing zero rows (for this we write $\overline{\mathrm{HNF}}$), 
they have sizes $6 \times 18$, $6 \times 18$, $6 \times 12$; see Figure \ref{HNFY}.
We write
  \[
  \overline{\mathrm{HNF}}(Y) = 
  \left[
  \begin{array}{ccc}
  \overline{\mathrm{HNF}}\big( \left[ \begin{array}{@{}c@{\;}c@{\;}c@{}} X_1 & X_5 & X_6 \end{array} \right] \big) & 0 & 0 
  \\
  0 & \overline{\mathrm{HNF}}\big( \left[ \begin{array}{@{}c@{\;}c@{\;}c@{}} X_2 & X_4 & X_8 \end{array} \right] \big) & 0
  \\
  0 & 0 & \overline{\mathrm{HNF}}\big( \left[ \begin{array}{@{}c@{\;}c@{}} X_3 & X_7 \end{array} \right] \big)
  \end{array}
  \right]
  \]  
We sort the rows of $\overline{\mathrm{HNF}}(Y)$ so that all the rows containing $q$ come first, and 
apart from this, rows with fewer nonzero entries come first.
We retain only those rows which do not belong to the $S_3$-submodule generated by the previous rows, 
and obtain the following relations.

\begin{theorem}
The following three relations define a one-parameter deformation of the nonsymmetric dendriform operad $\mathsf{Dend}$
into the category of symmetric operads:
  \begin{align*}
  &
  (q+3) 
  \big[ 
  ( ( a \succ b ) \prec c ) - ( a \succ ( b \prec c ) ) 
  \big]
  \\[-1mm]
  &\qquad
  +
  (q-1)
  \big[
  ( a \succ ( c \prec b ) ) - ( b \succ ( a \prec c ) )  
  + ( b \succ ( c \prec a ) ) - ( c \succ ( a \prec b ) )
  \big]
  = 0, 
  \\ 
  &
  (q+3)
  \big[
  ( ( a \prec b ) \succ c ) + ( ( a \succ b ) \succ c ) - ( a \succ ( b \succ c ) ) 
  \big]
  \\[-1mm]
  &\qquad
  +
  (q-1)
  \big[
  ( a \succ ( c \succ b ) ) - ( b \succ ( a \succ c ) )  
  + ( b \succ ( c \succ a ) ) - ( c \succ ( a \succ b ) )  
  \big]
  = 0,  
  \\
  &
  (q+3)
  \big[
  ( ( a \prec b ) \prec c ) - ( a \prec ( b \prec c ) ) - ( a \prec ( b \succ c ) )
  \big]
  \\[-1mm]
  &\qquad
  +
  (q-1)
  \big[
  ( a \prec ( c \prec b ) ) - ( b \prec ( a \prec c ) )  
  + ( b \prec ( c \prec a ) ) - ( c \prec ( a \prec b ) )  
  \big]
  \\[-1mm]
  &\qquad
  +
  (q-1)
  \big[
  ( a \prec ( c \succ b ) ) - ( b \prec ( a \succ c ) )  
  + ( b \prec ( c \succ a ) ) - ( c \prec ( a \succ b ) )  
  \big]
  =  0.
  \end{align*}
For $q = 1$ we obtain (4 times) the nonsymmetric dendriform relations.
\end{theorem}


\section{Dualizing the split polarization: deformation of diassociative}

The original references on Koszul duality for operads are \cite{GetzlerJones1994,GinzburgKapranov1994}.
To prepare for computing the Koszul dual of the deformed dendriform operad,
we multiply each column of $X$ by the sign of the corresponding permutation of $a,b,c$ 
and then multiply columns 25--48 (which have the second placement of parentheses) by $-1$;
see \cite{Loday2001}, \cite[Ch.~7]{LodayVallette2012} for details.
We denote the resulting ``signed'' matrix by $X'$.

As before, we partition $X'$ into $18 \times 6$ blocks $X'_1,\dots,X'_8$ corresponding 
to the association types, and rearrange them following the dendriform splitting into 
the $54 \times 48$ block diagonal matrix $Y'$:
  \[
  Y' =
  \left[
  \begin{array}{ccc}
  \left[ \begin{array}{ccc} X'_1 & X'_5 & X'_6 \end{array} \right] & 0 & 0 
  \\
  0 & \left[ \begin{array}{ccc} X'_2 & X'_4 & X'_8 \end{array} \right] & 0
  \\
  0 & 0 & \left[ \begin{array}{ccc} X'_3 & X'_7 \end{array} \right]
  \end{array}
  \right]
  \]
This implies a certain permutation $\xi \in S_{48}$ of the basis monomials of $\mathcal{SO}_2(3)$ from
their original order, of which we must keep track in order to undo it later.

We compute an invertible $48 \times 48$ integer matrix $U$ for which $U(Y')^t$ is the HNF of
$(Y')^t$ over $\mathbb{F}[q]$, and recalling that $Y'$ has rank 18, we define $N$ to be the matrix consisting of 
the bottom 30 rows of $U$.
This is the only point at which the computations become difficult; the entries of $N$ are polynomials in $q$
of degrees 0--8 with distribution $40, 22, 0, 7, 13, 123, 249, 120, 16$.
(We used the Maple package LinearAlgebra for this.)
For example, one of the entries of degree 8 is
  \[
  -\tfrac{1}{32768}
  \big(
  15q^8+550q^7+2386q^6+5230q^5+5184q^4-14894q^3-16850q^2+5018q+13361
  \big).
  \]
We compute the HNF of $N$; its nonzero entries are $\pm 1$, $\pm (q{-}1)$, $\pm (q{+}3)$: a remarkable improvement.
Moreover, the number of nonzero entries in each row is 2, 4 or 6.
We sort the rows as before (rows containing $q$ first, then by increasing number of nonzero entries), 
extract the rows which do not belong to the $S_3$-submodule generated by the previous rows, and
obtain the following relations.

\begin{theorem}
The following five relations define a one-parameter deformation of the nonsymmetric diassociative operad $\mathsf{Dias}$ 
into the category of symmetric operads:
  \begin{align*}
  &
  (q+3) 
  \big[
  ( ( a \succ b ) \prec c ) - ( a \succ ( b \prec c ) )
  \big]
  \\[-1mm]
  &\qquad
  +
  (q-1)
  \big[   
  ( a \succ ( c \prec b ) ) - ( b \succ ( a \prec c ) )  
  + ( b \succ ( c \prec a ) ) - ( c \succ ( a \prec b ) ) 
  \big] = 0,
  \\ 
  &
  (q+3)   
  \big[
  ( ( a \succ b ) \succ c ) - ( a \succ ( b \succ c ) )  
  \big]
  \\[-1mm]
  &\qquad
  +
  (q-1)
  \big[   
  ( a \succ ( c \succ b ) ) - ( b \succ ( a \succ c ) )  
  + ( b \succ ( c \succ a ) ) - ( c \succ ( a \succ b ) )  
  \big] = 0,
  \\ 
  &
  (q+3)   
  \big[
  ( ( a \prec b ) \prec c ) - ( a \prec ( b \succ c ) )  
  \big]
  \\[-1mm]
  &\qquad
  +
  (q-1)
  \big[   
  ( a \prec ( c \succ b ) ) - ( b \prec ( a \succ c ) )  
  + ( b \prec ( c \succ a ) ) - ( c \prec ( a \succ b ) )  
  \big] = 0,
  \\ 
  &
  ( a \prec b ) \succ c - ( a \succ b ) \succ c = 0, 
  \\
  &
  a \prec ( b \prec c ) - a \prec ( b \succ c ) = 0.
  \end{align*}
For $q = 1$ we obtain (4 times) the nonsymmetric diassociative relations.
Note that the associativities deform but the bar relations do not.
\end{theorem}


\end{document}